\documentclass[11pt,a4paper]{amsart}
\usepackage{amssymb,xspace}
\usepackage{amstext}
\theoremstyle{plain}
\usepackage{amsbsy,amssymb,amsfonts,latexsym}
\usepackage[all]{xy}

\usepackage{enumerate}
\usepackage{graphics}

\title{$\varepsilon$-hypercyclic operators that are not $\delta$-hypercyclic for $\delta<\varepsilon$}
\author{Fr\'ed\'eric Bayart}
\address{Laboratoire de Math\'ematiques Blaise Pascal UMR 6620 CNRS, Universit\'e Clermont Auvergne, Campus universitaire C\'ezeaux, 3 place Vasarely, 63178 Aubière Cedex, France.}
\email{frederic.bayart@uca.fr}



\newcommand{\veps}{\varepsilon}

\def\RR{\mathbb R}

\def\NN{\mathbb N}

\def\CC{\mathbb C}

\newtheorem{theorem}{Theorem}[section]

\newtheorem{lemma}[theorem]{Lemma}

\newtheorem{corollary}[theorem]{Corollary}

{\theoremstyle{definition}}
{\theoremstyle{definition}}

{\theoremstyle{definition}}

{\theoremstyle{definition}\newtheorem{definition}[theorem]{Definition}}

{\theoremstyle{definition}}

{\theoremstyle{definition}\newtheorem{remark}[theorem]{Remark}}

\newtheorem{question}[theorem]{Question}

\begin{document}

\begin{abstract}
For every fixed $\varepsilon\in (0,1)$,  we construct an operator on the separable
Hilbert space which is $\delta$-hypercyclic for all $\delta\in(\veps,1)$ and which is not $\delta$-hypercyclic for all $\delta\in(0,\veps)$.
\end{abstract}

\maketitle

\section{Introduction}

Let $X$ be a separable infinite dimensional Banach space. During the last decades the properties of the orbits of operators acting on $X$ have been widely studied. In particular, the notion of hypercyclic operators, namely operators with a dense orbit, has drawn the attention of many mathematicians (see for instance \cite{BM09}). It seems natural in this context to investigate operators having orbits with a property slightly weaker than denseness. Does this imply that the operator admits a dense orbit?
For instance, N. Feldman has shown in \cite{Fel02b} that if there is an orbit of $T\in\mathcal L(X)$ which meets every ball of radius $d>0$, then $T$ is hypercyclic. 

The following definition concerning operators admitting an orbit which intersects every cone of aperture $\veps$ has been introduced in \cite{BGM09}.

\begin{definition}
Let $\veps\in(0,1)$. A vector $x\in X$ is called an $\veps$-\emph{hypercyclic vector} for $T\in\mathcal L(X)$ provided
for every non-zero vector $y\in X$, there exists an integer $n\in \mathbb N$ such that $\|T^n x-y\|\leq\veps\|y\|$.
The operator $T$ is called $\veps$-\emph{hypercyclic} if it admits an $\veps$-hypercyclic vector.
\end{definition}

In \cite{BGM09}, the authors have shown that for every $\veps\in(0,1)$, there exists an $\veps$-hypercyclic operator on $\ell^1(\NN)$ which is not hypercyclic.
This was refined in \cite{BAYEPSHC} and \cite{Tap22} where similar examples are given on $\ell^2(\NN)$ and on more general spaces. 
Moreover it is pointed out in \cite[Remark 4.7]{Tap22} that the $\veps$-hypercyclic operator which is considered in that paper is not even $\delta$-hypercyclic for some $\delta\in(0,\veps)$.

This leaves open the following natural question: let $X$ be a Banach space, let $0<\delta<\veps<1$. Can we distinguish the class of $\delta$-hypercyclic operators
and that of $\veps$-hypercyclic operators acting on $X$? We give a positive answer for a large class of Banach spaces. To state our result we recall some terminology.
Let $(e_n)_{n\geq 0}$ be a basis of $X$ (namely every $x\in X$ writes uniquely $\sum_{n\geq 0} x_ne_n$) and let $C\geq 1$. 
We say that $(e_n)_{n\geq 0}$ is $C$-unconditional if for any $N\geq 0$, for any finite sequences of scalars $(a_n)_{n=0,\dots,N}$ and $(b_n)_{n=0,\dots,N}$ such that 
$|b_n|\leq |a_n|$ for all $n=0,\dots,N$, then 
$$\left\|\sum_{n=0}^N b_n e_n\right\|\leq C\left\|\sum_{n=0}^N a_n e_n\right\|.$$
Let us fix now $X$ and $Y$ two Banach spaces and suppose that $(f_n)$ is a $1$-unconditional basis of $Y$. We denote by $\bigoplus_Y X$ the vector space 
$$\bigoplus_Y X:=\left\{(x_n)\in X^{\NN}:\ \sum_{n=0}^{+\infty}\|x_n\|_X f_n\in Y\right\}$$
and we endow it by 
$$\|(x_n)\|=\left\|\sum_{n=0}^{+\infty}\|x_n\|_X f_n\right\|_Y.$$
It is standard that $\bigoplus_Y X$ is a Banach space.

Our main theorem now reads.

\begin{theorem}\label{THM:MAIN}
 Let $X$ be an infinite dimensional separable Banach space with a $1$-uncon\-ditional basis,
 let $Y$  be an infinite dimensional separable Banach space with a normalized $1$-unconditional basis
 such that the associated backward shift operator is continuous. For all $\veps\in(0,1)$,
 there exists an operator on $Z=\bigoplus_Y X$ which is not $\delta$-hypercyclic
 for all $\delta\in (0,\veps)$ and which is $\delta$-hypercyclic for all $\delta\in(\veps,1)$.
\end{theorem} 

Observe that if $X$ is either $c_0(\NN)$ or $\ell^p(\NN)$, $p\in[1,+\infty)$, then $X$
is isometric to $\bigoplus_X X$ by using the canonical basis of $X$. Therefore, it satisfies the assumptions of the previous theorem.
Recall also that if $T$ is $\veps$-hypercyclic for all $\veps\in(0,1)$, then it is hypercyclic (see \cite[Theorem 1.3]{BGM09}).

We will need a way to prove that an operator is $\veps$-hypercyclic. We state here a variant of the $\veps$-hypercyclicity criterion given in  \cite[Theorem 1.2]{Tap22}.
Its proof is completely similar.

\begin{theorem}\label{thm:epshccriterion}
 Let $X$ be an infinite dimensional separable Banach space, let $T\in\mathcal L(X)$ and let $\veps\in(0,1)$. 
 Assume that there exist a dense subset $\mathcal D$ of $X$, a sequence $(u(k))$ dense in $X$ such that that, for all $k\geq 0,$
 $u(k)=u(l)$ for infinitely many integers $l$, a sequence $(v(k))$ of vectors in $X$ and an increasing sequence
 of positive integers $(n_k)$ such that 
 \begin{itemize}
  \item $\lim_{k\to+\infty}\|T^{n_k}x\|=0$ for all $x\in\mathcal D$;
  \item $\lim_{k\to+\infty} \|v(k)\|=0;$
  \item for all $k\geq 0$, $\|T^{n_k}v(k)-u(k)\|\leq \veps \|u(k)\|.$
 \end{itemize}
 Then $T$ is $\delta$-hypercyclic for all $\delta>\veps.$ 
 \end{theorem}

 The remaining part of the paper is devoted to the proof of Theorem \ref{THM:MAIN}.

\section{Proofs}

\subsection{A geometric lemma in dimension 2}

The construction ultimately relies on the following fact regarding normed spaces of dimension 2.
It deals with the distance of some fix vector to lines depending on a parameter.

\begin{lemma}\label{LEM:GEO2}
 Let $F$ be a normed space of dimension $2$, let $(u,v)$ be a normalized basis of $F$, let $(u^*,v^*)$ be the dual basis
 and assume that $\|u^*\|=\|v^*\|=1$. For all $\veps\in(0,1)$, there exists $\omega\in[\veps,\veps(1-\veps)^{-1}]$ such that
 $$\min_{y\in\mathbb C} \|(y-1)u+y\omega v\|=\veps.$$
\end{lemma}

\begin{proof}
 When $\omega=\veps$, 
 $$\min_{y\in\CC}\|(y-1)u+y\omega v\|\leq \|\omega v\|\leq \veps.$$
 When $\omega=\veps(1-\veps)^{-1}$, for all $y\in\CC$,
 $$\|(y-1)u+y\omega v\|\geq \max(|y-1|,|y|\omega).$$
 Now, if $|y|\geq1/(1+\omega)$, $|y|\omega\geq \omega/(1+\omega)\geq\veps$ and if $|y|\leq 1/(1+\omega)$, 
 $$|y-1|\geq 1-\frac 1{1+\omega}= \veps.$$
 Therefore, $\min_{y\in\CC}\|(y-1)u+y\omega v\|\geq\veps$. The result follows by continuity
 of $\omega\mapsto \min_{y\in\CC}\|(y-1)u+y\omega v\|.$
\end{proof}

\begin{remark}
 If $\|au+bv\|=(|a|^p+|b|^p)^{1/p}$ for some $p\in(1,+\infty)$, then it is easy to prove that the value of 
$\omega$ is given by
$$\frac{\omega}{(1+\omega^{\frac p{p-1}})^{\frac{p-1}p}}=\veps$$
and that the minimum is attained at
$$y=\frac{1}{1+\omega^{\frac p{p-1}}}.$$
When $p=1,$ $\omega=\veps$ and $y=1$. When $p=\infty,$ $\omega=\frac{\veps}{1-\veps}$ and $y=\frac1{1+\omega}=1-\veps.$
This corresponds to the extremal cases of Lemma \ref{LEM:GEO2}.
\end{remark}

\subsection{The construction of a sequence of operators on $X$}
As the previous constructions of $\veps$-hypercyclic operators which are not $\delta$-hypercyclic, 
our operator will be an operator weighted shift. The next part of the proof consists in defining his weights.
We denote by $(e_n)$ (resp. $(f_n)$) the $1$-unconditional basis of 
$X$ (resp. $Y$) which appears in the statement of Theorem \ref{THM:MAIN}.
We may assume that $(e_n)$ is normalized which implies (by $1$-unconditionality) that $(e_n^*)$ is normalized too. 

The strategy is the following. At each step $k$ we will define weights $A_{m_k+1},\cdots,A_{m_{k+1}}$ 
such that the products $A_{m_k+1}\cdots A_j$, $j=m_k+1,\dots,m_{k+1}$ leave $e_0$ invariant, send $e_k$ onto the line defined by Lemma \ref{LEM:GEO2} and $e_l$
onto a multiple of $e_l$ for $l\neq k.$ Therefore, provided $e_0^*(u)$ is small, $A_{m_k+1}\cdots A_{j}(u)$ can be close to $e_0,$ but not too close. 

We proceed with the details. We set
$$\lambda=\frac{3}{\veps(1-\veps)}\textrm{ and }\kappa= (1+\lambda)+\max(1+\veps(1-\veps)^{-1},2/\veps).$$
We exhibit two sequences of integers $(m_k)_{k\geq 1}$ and $(r_k)_{k\geq 1}$ and a sequence
of operators $(A_j)_{j\geq 1}$ on $X$ such that, for all $k\geq 1$, 
\begin{enumerate}[(i)]
 \item $A_n e_0=e_0$ for all $n=m_k+1,\dots,m_{k+1}$;
 \item $A_n$ is invertible, $\|A_n\|\leq \kappa$ for all  $n=m_k+1,\dots,m_{k+1}$;
 \item $A_{m_k+1}\cdots A_{m_{k+1}}=\textrm{Id}$.
% \item $A_{m_k+1}\cdots A_{m_k+r_k+j}=2^{r_k}\textrm{Id}$ on $\textrm{span}(e_1,\dots,e_{k-1})$ for all $j=0,\dots,k-1$;
%  \item $m_k\geq k-1$, $m_{k+1}\geq m_k+r_k+k$ and setting $C_k=\sup_{j=0,\dots,k-1} \|A_{m_k}^{-1}\cdots A_{j+1}^{-1}\|$, 
% $$2^{r_k}\geq k^3 C_k;$$ 
 %\item there exists $x_k\in\CC$ with $|x_k|\leq k^{-3}C_k^{-1}$ such that, for all $j=0,\dots,k-1$, 
% $$\|A_{j+1}\cdots A_{m_k+r_k+j}(x_ke_k)-e_0\|=\veps;$$
% \item for all $w\in X$ with $e_0^*(w)=0,$ for all $n=m_k+1,\cdots,m_{k+1}$, 
% $$\|A_{m_k+1}\cdots A_n(w)-e_0\|\geq\veps.$$
\end{enumerate}
We initialize the construction by setting $m_1=0$. We assume that the construction has been done until step $k-1$
to do it at step $k \geq 1$. 
We thus have to define $m_{k+1}$, $r_k$ and $(A_j)_{j=m_k+1,\dots,m_{k+1}}$. 
We set $F_k=\textrm{span}(e_0,e_k)$ and $G_k=\overline{\textrm{span}}(e_l:\ l\neq 0,k)$
so that $X=F_k\oplus G_k$. Let $\omega_k \in [\veps,\veps(1-\veps)^{-1}]$ be given by Lemma \ref{LEM:GEO2} for $F=F_k$
and let $y_k\in\CC$ minimizing $y\mapsto \|(y-1)e_0+y\omega_k e_k\|$.
Since $(e_n)$
is a $1$-unconditional basis of $X$, we deduce from the definition of $\omega_k$ and $y_k$ that
\begin{equation}\label{eq:main1}
 \min_{y\in\CC,\ w\in G_k} \|(y-1)e_0+y\omega_k e_k+w\|=\|(y_k-1)e_0+y_k\omega_k e_k\|=\veps.
\end{equation}
Let $r_k>0$ be a very large integer (more precise conditions on $r_k$ will be given later) and let us set $m_{k+1}=m_k+r_k+k+1.$ For $j=1,\dots,r_k+k+1$, we define $A_{m_k+j}$ by
\begin{itemize}
 \item $A_{m_k+j}(e_0)=e_0.$
 \item 
 $$\left\{\begin{array}{rcl}
           A_{m_k+1}(e_k)&=&e_0+\omega_k e_k\\
           A_{m_k+2}(e_k)&=&\cdots\ =\ A_{m_k+r_k}(e_k)=2e_k\\
           A_{m_k+r_k+1}(e_k)&=&\cdots\ =A_{m_k+r_k+k-1}(e_k)=e_k\\
           A_{m_k+r_k+k}(e_k)&=&\frac{1}{2^{r_k-1}}e_k\\[0.2cm]
           A_{m_k+r_k+k+1}(e_k)&=&-\frac 1{\omega_k}e_0+\frac 1{\omega_k} e_k.
          \end{array}\right.$$
\item for $l\neq 0,k,$
$$\left\{\begin{array}{rcl}
A_{m_k+j}(e_l)&=&\lambda e_l,\ j=1,\dots,r_k+k,\\
A_{m_k+r_k+k+1}(e_l)&=&\frac1{\lambda^{r_k+k}}e_l.
\end{array}\right.$$
\end{itemize}

The invertibility of each $A_n$ comes from the invertibility of its restriction to $F_k$ and to $G_k$. Furthermore we prove $\|A_n\|\leq\kappa$. 
For $n=m_k+1,\dots,m_{k+1}$, for $a,b\in\CC$ and $w\in G_k$, 
\begin{align*}
 \big\|A_{n}(ae_0+be_k+w)\big\|&\leq |a|+|b|\max(1+\omega_k,2,2/\omega_k)+\lambda\|w\|\\
 &\leq \kappa \|ae_0+be_k+w\|
\end{align*}
where we have taken into account that $\omega_k\in[\veps,\veps(1-\veps)^{-1}]$.

To go further with the properties of $(A_j)$ we need to compute $A_{m_k+1}\cdots A_{m_k+j}e_k$
for $j=1,\dots,r_k+k+1$. We find
$$   A_{m_k+1}\cdots A_{m_k+j}e_k= \left\{\begin{array}{ll}
                                           e_0+\omega_k e_k&j=1\\
                                           2^{j-1}e_0+2^{j-1}\omega_k e_k&j=2,\dots,r_k\\
                                           2^{r_k-1}e_0+2^{r_k-1}\omega_k e_k&j=r_k+1,\dots,r_k+k-1\\
                                           e_0+\omega_k e_k&j=r_k+k\\
                                           e_k&j=r_k+k+1.
                                          \end{array}\right.$$
We then deduce the following formula, which will be equally important:
$$A_{m_k+j}^{-1}\cdots A_{m_k+1}^{-1}e_k=      
\left\{\begin{array}{ll}
                                           -\frac 1{\omega_k}e_0+\frac1{\omega_k} e_k&j=1\\[0.2cm]
                                            -\frac 1{\omega_k}e_0+ \frac 1{2^{j-1}\omega_k} e_k&j=2,\dots,r_k\\[0.2cm]
                                            -\frac 1{\omega_k}e_0+\frac 1{2^{r_k-1}\omega_k} e_k&j=r_k+1,\dots,r_k+k-1\\[0.2cm]
                                            -\frac 1{\omega_k}e_0+\frac1{\omega_k} e_k&j=r_k+k\\[0.2cm]
                                           e_k&j=r_k+k+1.
                                          \end{array}\right.$$                                    
%We then set $x_k=y_0/{2^{r_k-1}}$ so that $|x_k|\leq k^{-3}C_k^{-1}$. Moreover, for $j=0,\dots,k-1,$
%\begin{align*}
% A_{j+1}\cdots A_{m_k+r_k+j}(x_k e_k)-e_0&=A_j^{-1}\cdots A_1^{-1}\big(A_{m_k+1}\cdots A_{m_k+r_k+j}x_ke_k\big)-e_0\\
% &=A_j^{-1}\cdots A_1^{-1}(y_0e_0+y_0\omega e_k)-e_0\\
% &=(y_0-1)e_0+y_0\omega e_k
%\end{align*}
%since $e_0$ and $e_k$ are invariant by $A_1,\dots,A_j$, by the induction hypothesis (recall that $j\leq k-1$). 
%Hence, (vi) is satisfied. Pick now $w\in X$ with $e_0^*(w)=0$. Writing $w=xe_k+w'$ with $w'\in G$, we observe that for all $n\in\{m_k+1,\dots,m_{k+1}-1\}$,
%$$A_{m_k+1}\cdots A_n(w)-e_0=(y_n-1)e_0+y_n\omega e_k+w_n$$
%for some $y_n\in\CC$ and some $w_n\in G$. Therefore by \eqref{eq:main1}, (vii) is verified (since it is trivially true for $n=m_k+1$
%by (iii)).

\subsection{The operator}

We now glue together these maps. We formally define $T$ on $Z=\bigoplus_Y X$ by
$$T(u_0,u_1,\dots)=(A_1 u_1,A_2u_2,\dots).$$
Let $K_1$ be the norm of the backward shift operator associated to $(f_n)$. Then for $u=(u_0,u_1,\dots)$,
\begin{align*}
 \|Tu\|&=\left\|\sum_{n=1}^{+\infty}\|A_n u_n\|_X f_{n-1}\right\|_Y\\
 &\leq K_1 \left\| \sum_{n=1}^{+\infty} \|A_n\|\cdot \|u_n\|_X f_n\right\|_Y\\
 &\leq K_1\kappa \|u\|
\end{align*}
which implies that $T$ is well defined and maps boundedly $Z$ into itself.

\subsection{$T$ is not $\delta$-hypercyclic for any $\delta\in(0,\veps)$}
By contradiction, assume that $T$ is $\delta$-hypercyclic for some $\delta\in(0,\veps)$ and let $u=(u_0,u_1,\dots)$ be a $\delta$-hypercyclic vector for $T$. 
Observe that $\|u_n\|\to 0$ so that $u_{n,0}:=e_0^*(u_n)\to 0$. Therefore it is possible to fix $K>0$ such that
$$|K-u_{n,0}|\veps>\delta K \textrm{ for any }n\geq 0.$$
We set $v=(Ke_0,0,\dots)$. Let $n\geq 1$ be such that $\|v-T^n u\|\leq \delta\|v\|$ and let $k\geq 1$ be such that
$n\in[m_k+1,m_{k+1}]$. Let us write $u_n=u_{n,0}e_0+w_n$ with $e_0^*(w_n)=0$. Then
by using (i) and (iii),
\begin{align*}
 \|v-T^n u\| &\geq \|Ke_0-A_1\cdots A_n(u_n)\|\\
 &\geq \|K e_0 -A_{m_k+1}\cdots A_n(u_{n,0}e_0+w_n)\|\\
 &\geq \| (K-u_{n,0})e_0-A_{m_k+1}\cdots A_n (w_n)\|.
\end{align*}
If $n=m_{k+1}$, then $A_{m_k+1}\cdots A_n(w_n)=w_n$ and 
$$\|v-T^n u\|\geq |K-u_{n,0}|\geq\veps |K-u_{n,0}|>\delta\|v\|.$$
If $n\neq m_{k+1},$ then $A_{m_k+1}\cdots A_n(w_n)=x_ne_0+x_n \omega_k e_k+w'_n$ for some $w'_n\in G_k$ and some $x_n\in\mathbb C.$ Therefore
\begin{align*}
\|v-T^n u\|&\geq \left\|\big((K-u_{n,0})-x_n\big)e_0+x_n\omega_k e_k\right\|\\
&\geq \veps|K-u_{n,0}|>\delta\|v\|
\end{align*}
where we have used \eqref{eq:main1}. In both cases, we find a contradiction.

\subsection{$T$ is $\delta$-hypercyclic for all $\delta\in(\veps,1)$}

Let $\delta\in(\veps,1)$ and let us prove that $T$ is $\delta$-hypercyclic by applying Theorem \ref{thm:epshccriterion}.
Let $(u(k))$ be a dense sequence in $Z$ such that each $u(k)$ may be written $u(k)=(u_0(k),\dots,u_{k-1}(k),0,\dots)$
with $u_j(k)\in\textrm{span}(e_0,\dots,e_{k-1})$ and $\|u_j(k)\|\leq k$. Moreover for any $k\geq 1$,
we assume that there exist infinitely many integers $\ell$ with $u(k)=u(\ell)$. 

We want to find a sequence of vectors $(v(k))$ in $Z$ and a sequence of integers $(n_k)$ such that
$\|v(k)\|\to 0$ and $\|T^{n_k}v(k)-u(k)\|\leq \veps \|u(k)\|$ for all $k\geq 1$. 
We will define $v(k)=(0,\dots,0,v_0(k),\dots,v_{k-1}(k),0,\dots)$ where $v_0(k)$ is at the $(m_k+r_k)$-th position. 
Let $k\geq 1,$ let $j\in\{0,\dots,k-1\}$ and let us write 
$$u_j(k)=\sum_{s=0}^{k-1}u_{j,s}(k)e_s.$$
Let $l$ be the unique integer such that $m_l\leq j<m_{l+1}.$ We will search $v_j(k)$ under the form
$$v_j(k)=\sum_{\substack{s=1\\s\neq l}}^{k-1}\frac{u_{j,s}(k)}{\lambda^{r_k+j}}\lambda^{j-m_l}e_s+xe_l+ye_k$$
where $x$ and $y$ will be chosen so that 
$$\left\|A_{j+1}\cdots A_{m_k+r_k+j}(v_j(k))-u_j(k)\right\|\leq\veps \|u_j(k)\|$$
and $\|v_j(k)\|\leq k^{-2}.$ Upon this has been done, we can easily apply Theorem \ref{thm:epshccriterion} to deduce that $T$ is $\delta$-hypercyclic for $\delta>\veps.$ Indeed,
$T$ has a dense generalized kernel and, for all $k\geq 1$,
\begin{align*}
\|v(k)\|&=\left\|\sum_{j=0}^{k-1}\|v_j(k)\| f_{m_k+r_k+j}\right\|\\
&\leq \sum_{j=0}^{k-1}\|v_j(k)\|\leq \frac 1k.
\end{align*}
Moreover 
$$T^{m_k+r_k}(v(k))=(A_1\cdots A_{m_k+r_k}(v_0(k)),\dots,A_k\cdots A_{m_k+r_k+k-1}(v_{k-1}(k)),0,\dots).$$
Therefore,
\begin{align*}
 \|u(k)-T^{m_k+r_k}(v(k))\|&=\left\|\sum_{j=0}^{k-1} \left\|A_{j+1}\cdots A_{m_k+r_k+j}(v_j(k))-u_j(k)\right\|f_j\right\|\\
 &\leq  \veps \left\|\sum_{j=0}^{k-1} \|u_j(k)\| f_j\right\|\\
 &\leq \veps \|u(k)\|.
\end{align*}
So let us compute $A_{j+1}\cdots A_{m_k+r_k+j}(v_j(k))=:z_j(k).$
\begin{align*}
z_j(k)&=A_j^{-1}\cdots A_1^{-1}A_1\cdots A_{m_k+r_k+j}(v_j(k))\\
&=A_j^{-1}\cdots A_{m_l+1}^{-1}A_{m_k+1}\cdots A_{m_k+r_k+j}(v_j(k))\\
&=A_j^{-1}\cdots A_{m_l+1}^{-1}\left(\sum_{\substack{s=1\\s\neq l}}^{k-1}\lambda^{j-m_l}u_{j,s}(k)e_s+\lambda^{r_k+j}xe_l+2^{r_k-1}ye_0+2^{r_k-1}y\omega_k e_k\right).
\end{align*} 
The easiest case is when $j=m_l.$ In that case, 
$$z_j(k)=2^{r_k-1}ye_0+\sum_{\substack{s=1\\s\neq l}}^{k-1}u_{j,s}(k)e_s+\lambda^{r_k+j}xe_l+2^{r_k-1}y\omega_k e_k.$$
We simply choose $x=\frac1{\lambda^{r_k+j}}u_{j,l}(k)$ and $y=\frac{y_k}{2^{r_k-1}}u_{j,0}(k)$ so that  by \eqref{eq:main1}
\begin{align*}
\|z_j(k)-u_j(k)\|&=|u_{j,0}(k)|\cdot \|(y_k-1)e_0+y_k\omega_k e_k\|\\
&\leq\veps |u_{j,0}(k)|\leq \veps \|u_j(k)\|
\end{align*}
whereas 
$$\|v_j(k)\|\leq \sum_{s=1}^{k-1}\frac{\|u_j(k)\|}{\lambda^{r_k}}+
\frac{|y_k| \cdot \|u_{j}(k)\|}{2^{r_k-1}}\leq k^{-2}$$
provided $r_k$ is sufficiently large.

Let us now turn to $j>m_l.$ In that case, there exists $0\leq t_j\leq j-m_l$ such that
\begin{align*}
z_j(k)&=\sum_{\substack{s=1\\s\neq l}}^{k-1}u_{j,s}(k)e_s+\frac{\lambda^{r_k+j}}{\omega_l 2^{t_j}} xe_l+\left(2^{r_k-1}y-\frac{\lambda^{r_k+j}x}{\omega_l}\right)e_0+\frac{2^{r_k-1}\omega_k y}{\lambda^{j-m_l}}e_k.
\end{align*}
If we set to simplify the notations
$$\begin{array}{rcllrcl}
\lambda_j&=&\lambda^{j-m_l}&\quad\quad&\mu_j&=&2^{t_j}\\
x'&=&\displaystyle \frac{\lambda^{r_k+j}}{\omega_l 2^{t_j}}x&&y'&=&\displaystyle\frac{2^{r_k-1}}{\lambda^{j-m_l}}y\\[0.2cm]
a&=&u_{j,0}(k)&&b&=&u_{j,l}(k)
\end{array}$$
then 
$$z_j(k)-u_j(k)=\big(\lambda_j y'-\mu_j x'-a\big)e_0+(x'-b)e_l+y'\omega_k e_k.$$
We are now ready to choose $x'$ and $y'$, namely $x$ and $y$. We indeed set
$$x'=b\textrm{ and }y'=\frac{a+\mu_j b}{\lambda_j}$$
so that
$$z_j(k)-u_j(k)=\left(\frac{a+\mu_j b}{\lambda_j}\right)\omega_k e_k.$$
Therefore, by $1$-unconditionality of $(e_k)$, since $t_j\leq j-m_l$, $\lambda\geq 2$ and $\omega_k\leq (1-\veps)^{-1}$,
\begin{align*}
\|z_j(k)-u_j(k)\|&\leq \frac{1+\mu_j}{\lambda_j}\omega_k\|u_j(k)\|\\
&\leq \left(\frac 1{\lambda}+\frac{2}{\lambda}\right)\omega_k\|u_j(k)\|\\
&\leq \frac{3}{\lambda(1-\veps)}\|u_j(k)\|\leq \veps\|u_j(k)\|
\end{align*}
by our choice of $\lambda.$ Finally observe that 
\begin{align*}
\|v_j(k)\|&\leq \sum_{\substack{s=1\\s\neq l}}^{k-1} \frac{\|u_j(k)\|}{\lambda^{r_k}}+\frac{\omega_l 2^{t_j}}{\lambda^{r_k}}\|u_j(k)\|+\frac{ \lambda^{j-m_l}}{2^{r_k-1}}\times\frac 3\lambda\|u_j(k)\|\\
&\leq k\left(\sum_{s=1}^{k-1}\frac1{\lambda^{r_k}}+\frac{\omega_l 2^{m_k}}{\lambda^{r_k}}+\frac{3\lambda^{m_k-1}}{2^{r_k-1}}\right)\leq k^{-2}
\end{align*}
provided $r_k$ has been chosen large enough.
%Let $j\in\{0,\dots,k-1\}$ and let us write 
%$$A_{m_k}^{-1}\cdots A_{j+1}^{-1}u_j(k)=u_{j,0}(k)e_0+w_j(k)$$
%with $w_j(k)\in\textrm{vect}(e_1,\dots,e_{k-1})$ and $|u_{j,0}(k)|,\ \|w_j(k)\|\leq C_k\|u_j(k)\|\leq C_k k$. 
%We set $n_k=m_k+r_k$ and for $j=0,\dots,k-1$, 
%$$v_j(k)=\frac{w_j(k)}{2^{r_k}}+u_{j,0}(k)x_ke_k.$$
%We finally define $v(k)=(0,\dots,0,v_0(k),\dots,v_{k-1}(k),0,\dots)$ where $v_0(k)$ is at the $m_k+r_k$-th position. 
%Observe that 
%\begin{align*}
% \|v(k)\|&=\left\|\sum_{j=0}^{k-1} \|v_j(k)\| f_{m_k+r_k+j}\right\|\\
% &\leq \sum_{j=0}^{k-1}\|v_j(k)\|\\
% &\leq \frac{k^2 C_k}{2^{r_k}}+k^2 C_k |x_k|\\
% &\leq \frac 2k\to 0.
%\end{align*}
%Moreover 
%$$T^{n_k}(v(k))=(A_1\cdots A_{m_k+r_k}(v_0(k)),\dots,A_k\cdots A_{m_k+r_k+k-1}(v_{k-1}(k)),0,\dots).$$
%Now, for $j=0,\dots,k-1,$ by (iv), the definition of $w_j(k)$ and lastly by (i),
%\begin{align*}
% A_{j+1}\cdots A_{m_k+r_k+j}(v_j(k))&=A_{j+1}\cdots A_{m_k}(w_j(k)+u_{j,0}(k)x_k A_{m_k+1}\cdots A_{m_k+r_k+j}e_k)\\\
% &=A_{j+1}\cdots A_{m_k}(A_{m_k}^{-1}\cdots A_{j+1}^{-1}(u_j(k))-u_{j,0}(k)e_0+\\
% &\quad\quad u_{j,0}(k)x_k A_{m_k+1}\cdots A_{m_k+r_k+j}e_k)\\
% &=u_j(k)-u_{j,0}(k)e_0+u_{j,0}(k)x_k A_{j+1}\cdots A_{m_k+r_k+j}e_k.
%\end{align*}
%Therefore
%\begin{align*}
% \|u(k)-T^{n_k}(v(k))\|&=\left\|\sum_{j=0}^{k-1} |u_{j,0}(k)| \cdot \|e_0-A_{j+1}\cdots A_{m_k+r_k+j}(x_ke_k)\|f_j\right\|\\
% &\leq  \veps \left\|\sum_{j=0}^{k-1} \|u_j(k)\| f_j\right\|\\
% &\leq \veps \|u(k)\|.
%\end{align*}
%
%ET LA IL Y A UNE VRAIE CONNERIE! Il n'y a pas de raisons que $|u_{j,0}(k)|\leq \|u_j(k)\|.$

\subsection{Concluding remarks and questions}

Observe that in Theorem \ref{THM:MAIN}, we assume a priori that $(f_n)$ is normalized which was not the case for $(e_n)$.
Indeed, if we normalize $(f_n)$, this could destroy the continuity of the associated backward shift operator.

Following \cite{Tap22}, we can slightly enlarge the scale of spaces where it is possible to produce such an example.
Indeed, observe that, by adjusting $(m_k)$ and $(r_k)$ during the construction (we may ask that $r_k$ is bigger
than any prescribed value), it is possible to ensure that $T$ satisfies the $\veps$-hypercyclicity criterion
with the sequence $(n_k)$ chosen as a subsequence of a prescribed sequence $(p_k)$. Arguing like in \cite[Theorem 4.10]{Tap22},
we get therefore the following corollary.
\begin{corollary}\label{cor:complemented}
 Let $Z$ be a separable Banach space. Assume that 
 $Z$ admits an infinite dimensional complemented subspace $\bigoplus_Y X$, where
 $X$ is an infinite dimensional separable Banach space with a $1$-unconditional basis and
 $Y$  is an infinite dimensional separable Banach space with a normalized $1$-unconditional basis
 such that the associated backward shift operator is continuous. Then for all $\veps\in(0,1)$,
 there exists an operator on $Z$ which is not $\delta$-hypercyclic
 for all $\delta\in (0,\veps)$ and which is $\delta$-hypercyclic for all $\delta\in(\veps,1)$.
\end{corollary}
Writing $V=\bigoplus_Y X$ and $Z=V\oplus W$, the main step is to define $T=T_1\oplus T_2$
where $T_2$ is a hypercyclic operator satisfying the hypercyclicity criterion and $T_1$ is the operator defined above.
In particular, if $Z$ contains a complemented copy of $c_0(\mathbb N)$ or of $\ell^p(\mathbb N)$, $p\in[1,+\infty)$,
then it satisfies the assumptions of Corollary \ref{cor:complemented}.

To conclude, we observe that we can give an additional property of our operator when it is defined on $\ell^1$.
\begin{theorem}\label{THM:CASEL1}
 Assume that $X=Y=\ell^1$. Then $T$ is not $\veps$-hypercyclic.
\end{theorem}
\begin{proof}
 By contradiction assume that $u$ is an $\veps$-hypercyclic vector for $T$. 
 Let us introduce $M=\max|u_{n,0}|$, $v=(e_0,0,\dots)$, $I=\left[\frac{2M}{1-\veps},+\infty\right)$
 and for $n\in\NN,$ 
 $$I_n=I\cap\left\{a\in\RR:\ \|T^n u-av\|\leq\veps a\right\}$$
 so that $I=\bigcup_n I_n.$ We first observe that if $n=m_k$ for some $k$, then $I_n$ is empty. 
 Indeed, for these values of $n$, since $A_1\cdots A_n u_n=u_n,$ we get
 \begin{align*}
  \|T^nu-av\|&\geq \|u_n-ae_0\|\\
  &\geq |u_{n,0}-a|\\
  &\geq a-\frac{1-\veps}2a\\
  &\geq \left(\frac 12+\frac{\veps}2\right)a>\veps a=\veps\|v\|.
 \end{align*}
Let $\mathcal N=\{n\in\NN:\ I_n\neq\varnothing\}$. For $n\in\mathcal N,$ let $a\in I_n.$ Since 
$n\neq m_k,$ we know that $A_1\cdots A_n(u_n)=(u_{n,0}+x_n)e_0+x_n\veps e_k+w_n$
for some $x_n\in\CC$ and some $w_n\in \ell^1$ with $e_0^*(w_n)=e_k^*(w_n)=0.$ Therefore
\begin{align*}
 \veps a&\geq \|T^n u-av\|\\
 &\geq |u_{n,0}+x_n-a|+|x_n|\veps.
\end{align*}
Arguing as above, we find $\Re e(x_n)\geq 0$ so that
$$\veps a\geq |\Re e(u_{n,0})+\Re e(x_n)-a|+\Re e(x_n)\veps.$$
We thus find 
$$
\left\{
\begin{array}{rcl}
 \veps a\geq \Re e(u_{n,0})+\Re e(x_n)-a+\Re e(x_n)\veps\\
 \veps a\geq a-\Re e(x_n)-\Re e(u_{n,0})+\Re e(x_n)\veps
\end{array}
\right.$$
which in turn yields
$$\Re e(x_n)+\frac{\Re e(u_{n,0})}{1+\veps}\leq a\leq \Re e(x_n)+\frac{\Re e(u_{n,0})}{1-\veps}.$$
In particular, $\Re e(u_{n,0})$ must be positive and $I_n$ is contained in an interval of length $c\Re e(u_{n,0})$
for some $c>0$. But since we are working on $\ell^1,$ $\sum_n |\Re e(u_{n,0})|<+\infty,$ which contradicts
that $I=\bigcup_n I_n$ has infinite length.
\end{proof}

\begin{corollary}
 Let $\veps\in(0,1)$. There exists an operator $T$ on $\ell^1$ which is $\delta$-hypercyclic operator for all $\delta\in(\veps,1)$
 and which is not $\veps$-hypercyclic.
 \end{corollary}

This leads to the following natural question:
 
\begin{question}
 Let $\veps\in(0,1)$. Does there exist an operator which is $\delta$-hypercyclic if and only if $\delta\in[\veps,1)$?
\end{question}

\end{document}